\documentclass[11pt,twoside]{amsart} 

\usepackage{latexsym}
\usepackage{amsmath}
\usepackage{amssymb}
\usepackage{amsthm}
\usepackage{amscd}
\usepackage{enumerate} 
\usepackage{amssymb} 
\usepackage{mathrsfs}

\def\C{{\mathscr{C}}}
\def\CC{{\mathbb{C}}}

\def\K{{\mathscr{K}}}

\def\PP{{\mathbb{P}}}
\def\U{{\mathscr{U}}}

\usepackage{latexsym}

\newtheorem{them}{Theorem}[section]

\newtheorem{pro}[them]{Proposition}

\newtheorem{exa}[them]{Example}

\newtheorem{lem}[them]{Lemma}

\newtheorem{rem}[them]{Remark}

\newtheorem{cor}[them]{Corollary}

\newtheorem{defi}[them]{Definition}

\newtheorem{NA}[them]{Notation-Assumptions}

\title[A bound of lengths of chains of minimal rational curves]{A bound of lengths of chains of minimal rational curves on Fano manifolds of Picard number $1$}

\author{Kiwamu Watanabe}

\date{January 2010, Revised: April 25 2011.}

\address{Department of Mathematical Sciences  School of Science and Engineering Waseda University, 
4-1 Ohkubo 3-chome 
Shinjuku-ku 
Tokyo 169-8555 
Japan}

\address{Current address: Graduate School of Mathematical Sciences, University of Tokyo, 3-8-1 Komaba Meguro-ku Tokyo 153-8914, Japan.}

\email{watanabe@ms.u-tokyo.ac.jp}


\begin{document}

\maketitle


\begin{abstract}
In this paper, we investigate the minimal length of chains of minimal rational curves needed to join two general points on a Fano manifold of Picard number $1$ under mild assumptions. In particular, we give a sharp bound of the length by a fundamental argument. As an application, we compute the length for Fano manifolds of dimension $\leq 7$.  
\end{abstract}

\section{Introduction}

We say that a complex projective manifold $X$ is {\it Fano} if its anticanonical divisor is ample.  Rational curves on Fano manifolds have been studied by several authors. For instance, J. Koll\'{a}r, Y. Miyaoka and S. Mori proved the following: 
\begin{them}[\cite{KMM1}, \cite{Na}]\label{KMM} For a Fano $n$-fold of Picard number $\rho=1$, two general points can be connected by a smooth rational curve whose anticanonical degree is at most $n(n+1)$. 
\end{them}
In \cite{KMM1}, they also remarked that their proof can be modified to improve it to a bound which is asymptotically $\frac{n^2}{4}$. 
As a consequence of Theorem~\ref{KMM}, we know the $n$-dimensional Fano manifolds of $\rho=1$ form a bounded family. In this direction, J. M. Hwang and S. Kebekus studied the minimal length of chains of minimal rational curves needed to join two general points \cite{HK}. In the previous article \cite{Wa}, we computed the minimal length in some cases. For example, we dealt with the case where the dimension of $X$ is at most $5$. As a corollary, we provided a better bound on the degree of Fano $5$-folds of $\rho=1$.

Let $X$ be a Fano $n$-fold of $\rho=1$, ${\rm RatCurves}^n(X)$ the normalization of the space of rational curves on $X$ (see \cite[II. Definition-Proposition~2.11]{Ko}) and $\K$ a {\it minimal rational component}, which is a dominating irreducible component of ${\rm RatCurves}^n(X)$ whose anticanonical degree is minimal among such families. As in \cite[Assumption~2.1]{HK}, assume that for general $x \in X$, 

\begin{enumerate}
\item $\K_x:=\{[C]\in \K|x \in C \}$ is irreducible, and
\item $p:=\dim \K_x >0$.
\end{enumerate}
Remark that all known examples with $p>0$ satisfy the first condition (\cite[Remark~2.2]{HK}). Furthermore if $p=0$, our problem dealing in this paper is easy (see Remark~\ref{remark} and \cite[Remark~2.2]{HK}). 

We denote by $l_{\K}$ the minimal length of chains of general $\K$-curves needed to join two general points (for a precise definition, refer to Definition~\ref{defil}). 
In this paper, we give a sharp bound of the length $l_{\K}$ by a fundamental argument under the mild assumptions $\rm (i)$ and $\rm (ii)$. Our main result is
\begin{them}\label{MT} Let $X$ be a Fano $n$-fold of $\rho=1$ and $\K$ a minimal rational component of $X$ such that $\K_x$ is irreducible of dimension $p>0$ for general $x \in X$. Then we have 
\begin{eqnarray} 
\lfloor \frac{n-1}{p+1}\rfloor +1 \leq l_{\K} \leq \lfloor \frac{n-p}{2}\rfloor +1, \nonumber 
\end{eqnarray}
where $\lfloor d \rfloor$ is the largest integer $\leq d$.
\end{them}

Remark that the lower bound comes from \cite[Proposition~2.4]{HK} (see Proposition~\ref{in}) directly. Our main contribution is to establish the sharp upper bound. As a byproduct, applying the argument of \cite[Proof of the Theorem~Step~3, Corollary~1]{KMM1}, this theorem implies the following:

\begin{cor}\label{co} Let $X$ be a Fano manifold as in Theorem~\ref{MT}. Then the following holds. 
\begin{enumerate}
\item Two general points on $X$ can be connected by a smooth rational curve whose anticanonical degree is at most $(p+2)(\lfloor \frac{n-p}{2}\rfloor +1) \leq \frac{(n+3)^2}{8}$. 
\item $(-K_X)^n \leq \{(p+2)(\lfloor \frac{n-p}{2}\rfloor +1)\}^n \leq \{\frac{(n+3)^2}{8}\}^n$, where $-K_X$ stands for the anticanonical divisor of $X$.
\end{enumerate}
\end{cor}

This paper is organized as follows: In Section $2$, we give a precise definition of the {\it length} of chains of minimal rational curves. In Section $3$, we give a proof of our main theorem via a fundamental approach. In Section $4$, we investigate Fano manifolds whose {\it varieties of minimal rational tangents} have low-dimensional secant varieties. In Section $5$, we study the lengths of Fano manifolds of dimension $\leq 7$ by applying some previous results. In this paper, we work over the complex number field.

\section{Definition of length}\label{Dl}

\begin{defi} \rm 
\begin{enumerate} 
\item By a {\it variety}, we mean an integral separated scheme of finite type over the complex number field. We call a $1$-dimensional proper variety a {\it curve}. A {\it manifold} means a smooth variety. 
\item For a rational curve $C$ on a manifold $X$, let $f:\PP^1 \rightarrow C \subset X$ be the normalization. Then $C$ is {\it free} if $f^*T_X$ is semipositive, where $T_X$ stands for the tangent bundle of $X$. 
\item For a projective variety $X$ and a rational curve $C$ on $X$, $C$ is {\it $\K$-curve} if $[C]$ is contained in a subset $\K \subset {\rm RatCurves}^n(X)$.
\item For a projective variety $X$ and an irreducible component $\K$ of ${\rm RatCurves}^n(X)$, $\K$ is a {\it dominating family} if for a general point $x \in X$ there exists a $\K_x$-curve. 
\item For a Fano manifold $X$, a {\it minimal rational component} means a dominating irreducible component of ${\rm RatCurves}^n(X)$ whose anticanonical degree is minimal among such families.
\item For a vector space $V$, $\PP(V)$ denotes the projective space of lines through the origin in $V$.
\end{enumerate}
\end{defi}

Except Theorem~\ref{ke} and Lemma~\ref{fl}, we always assume the following throughout this section.

\begin{NA}\label{NA} \rm 

Let $X$ be a Fano $n$-fold of $\rho=1$, $\K$ a minimal rational component of $X$ such that for general $x \in X$, 
\begin{enumerate}
\item $\K_x:=\{[C]\in \K|x \in C \}$ is irreducible, and
\item $p:=\dim \K_x >0$.
\end{enumerate}

Notice that $p$ does not depend on the choice of a minimal rational component $\K$. It is a significant invariant of $X$. 

Let $\pi: \U \rightarrow \K$ and $\iota: \U \rightarrow X$ be the associated universal morphisms. 
Remark that $\pi$ is a {\it $\PP^1$-bundle} in the sense of \cite[II. Definition~2.5]{Ko}, that is, smooth, proper and for every $z \in \K$ the fiber $\pi^{-1}(z)$ is a rational curve (\cite[II. Corollary~2.12]{Ko}).
\end{NA}

\begin{them}[{\cite[Theorem~3.3]{Ke2}}]\label{ke}
Let $X$ be a normal projective variety and $\K \subset {\rm RatCurves}^n(X)$ a dominating family of rational curves of minimal degrees. Then, for a general point $x \in X$, there are only finitely many $\K_x$-curves which are singular at $x$.\end{them}

\begin{rem}\rm The original statement of Theorem~\ref{ke} is proved under much weaker assumptions. For detail, see \cite[Theorem~3.3]{Ke2}. 
\end{rem}

Let $X$ and $\K$ be as in the Notation-Assumptions~\ref{NA}. 
From Theorem~\ref{ke} and a well-known argument similar to the one used in the proof of \cite[II. Theorem~3.11]{Ko}, we know there exists a non-empty open subset $X^0 \subset X$ satisfying 

\begin{enumerate}
\item any $\K$-curve meeting $X^0$ is free, and
\item for any $x \in X^0$, there are only finitely many $\K_x$-curves which are singular at $x$.
\end{enumerate}

Here $\K^0:=\pi(\iota^{-1}(X^0)) \subset \K$ and $\U^0:=\pi^{-1}(\K^0) \subset \U$ are open subsets. Then we have the universal family of $\K^0$, that is, $\pi_0:=\pi|_{\U^0}: \U^0 \rightarrow \K^0$ and $\iota_0:=\iota|_{\U^0}: \U^0 \rightarrow X$. Since any $\K^0$-curve is free, $\iota_0: \U^0 \rightarrow X$ is smooth (see \cite[II. Theorem~2.15, Corollary~3.5.3]{Ko}). 

\begin{lem}\label{irr} For general $x \in X^0$, ${\iota_0}^{-1}(x)$ is irreducible.
\end{lem}

\begin{proof} For a general point $x \in X^0$, we have a surjective morphism $\pi : {\iota_0}^{-1}(x) \rightarrow \K_x$. The smoothness of $\iota_0 :\U^0 \rightarrow X$ implies that ${\iota_0}^{-1}(x)$ is equidimensional. Since $\K_x$ is irreducible of positive dimension and there are only finitely many $\K_x$-curves which are singular at $x$, ${\iota_0}^{-1}(x)$ is irreducible.
\end{proof}

Replacing $X^0$ with a smaller open subset of $X$, we may assume 

\begin{enumerate} 
\item[\rm (iii)] ${\rm for~any}~x \in X^0, {\iota_0}^{-1}(x)~{\rm is~irreducible.}$
\end{enumerate}

\begin{defi}\rm For general $x \in X^0$, define inductively 
\begin{enumerate}
\item $V_x^0:=\{x\}$, and 
\item $V_x^{m+1}:={\iota_0}({\pi_0}^{-1}({\pi_0}({\iota_0}^{-1}(V_x^m \cap X^0))))$.
\end{enumerate}
\end{defi}

\begin{lem}\label{fl} Let $f:X \rightarrow Y$ be a flat morphism between varieties with irreducible fibers and $W$ an irreducible constructible subset of $Y$. Then $f^{-1}(W)$ is irreducible.
\end{lem}

\begin{proof} This is a well-known fact. For instance, see \cite[Lemma~5.3]{De}.\end{proof}

Let consider $W^m_x:={\iota_0}^{-1}(V_x^m \cap X^0)$ and $\widetilde{W^m_x}:={\pi_0}^{-1}({\pi_0}({\iota_0}^{-1}(V_x^m \cap X^0)))$. 

\begin{lem}\label{dim} For general $x \in X^0$, the following holds.
\begin{enumerate}
\item $V_x^m$, $W^m_x$ and $\widetilde{W^m_x}$ are irreducible constructible subsets.
\item If $\dim V_x^m = \dim V_x^{m+1}$, we have $\dim V_x^m=n$. 
\item If $\dim V_x^m < n$, we have $\dim W^m_x=\dim V_x^m + p$ and $\dim \widetilde{W^m_x}=\dim V_x^m + p+1$.
\end{enumerate}
\end{lem}

\begin{proof} $\rm (i)$ We prove by induction on $m$. When $m=0$, $V_x^0=\{ x \}$ is irreducible. Assume that $V_x^m$ is irreducible. Remark that ${\iota_0}:\U^0 \rightarrow X$ and ${\pi_0}: \U^0 \rightarrow \K^0$ are flat. Since ${\iota_0}^{-1}(x)~{\rm is~irreducible}$ for any $x \in X^0$, we know $W^m_x$ and $\widetilde{W^m_x}$ are irreducible from Lemma~\ref{fl}. Because $V_x^{m+1}={\iota_0} (\widetilde{W^m_x})$, $V_x^{m+1}$ is also irreducible. Hence ${\rm (i)}$ holds. \\ 
$\rm (ii)$ This is in \cite{KMM1}. For the reader's convenience, we recall their proof. First assume that there exists a rational curve $[C] \in \K^0$ which is not contained in $\overline{V_x^m}$ satisfying $C \cap (V_x^m \cap X^0) \neq \emptyset$. Then $V_x^m$ is a proper subset of $V_x^{m+1}$. This implies that $\dim V_x^m < \dim V_x^{m+1}$. Hence, if $\dim V_x^m = \dim V_x^{m+1}$, every $\K^0$-curve meeting $V_x^m \cap X^0$ is contained in $\overline{V_x^m}$. Assume that $\dim V_x^m = \dim V_x^{m+1}$ for general $x \in X$. Let $q$ be the codimension of $V_x^m$ in $X$ and $T \subset X^0$ a sufficiently general $(q-1)$-dimensional subvariety. Denote $\bigcup_{x \in T}(V_x^m \cap X^0)$ by $H^0$ and its closure by $H$. Since the Picard number of $X$ is $1$, $H$ is an ample divisor on $X$. A general member $[C] \in \K^0$ is not contained in $H$. So we have $C \cap H^0 = \emptyset$. On the other hand, we see that $C \cap (H \setminus H^0) = \emptyset$. This follows from \cite[II. Proposition~3.7]{Ko}. It concludes that $C \cap H$ is empty. However this contradicts the ampleness of $H$. \\
$\rm (iii)$ Since ${\iota_0} : \U^0 \rightarrow X$ is flat, $W^m_x \rightarrow V_x^m \cap X^0$ is a flat morphism with irreducible fibers. This implies that $\dim W^m_x=\dim V_x^m + p$. Since ${\pi_0}: \U^0 \rightarrow \K^0$ is a $\PP^1$-bundle, $\dim \widetilde{W^m_x}=\dim W^m_x$ or $\dim W^m_x+1$. If the former equality holds, we have $\overline{W_x^m} \cap \widetilde{W_x^m} = \widetilde{W_x^m}$ in $\U^0$. Here, for a subset $A \subset \U^0$, denote by $\overline{A}$ the closure of $A$ in $\U^0$. Furthermore we see that 
\begin{equation*}
\iota_0(\widetilde{W^m_x})=\iota_0(\overline{W_x^m} \cap \widetilde{W_x^m}) \subset \overline{\iota_0(W_x^m)} \cap \iota_0(\widetilde{W_x^m}) \subset \iota_0(\widetilde{W_x^m}).
\end{equation*}
This yields that $\iota_0(\widetilde{W_x^m})=\overline{\iota_0(W_x^m)} \cap \iota_0(\widetilde{W_x^m}).$ Hence $\overline{\iota_0(\widetilde{W_x^m})}=\overline{\iota_0(W_x^m)}$. This concludes that $\dim V_x^{m+1}=\dim \iota_0(\widetilde{W_x^m})=\dim {\iota_0(W_x^m)}=\dim V_x^m$. This contradicts the assumption $\dim V_x^m < n$. Thus we have $\dim \widetilde{W^m_x}=\dim W^m_x+1=\dim V_x^m + p+1$.
\end{proof}

\begin{defi}\rm For general $x \in X^0$, we denote the dimension of $V_x^m$ by $d_m$. This definition does not depend on the choice of general $x \in X^0$.
\end{defi}

\begin{pro}[{\cite[Proposition~2.4]{HK}}]\label{in} 
\begin{enumerate} 
\item $d_1=p+1$, and
\item $d_{m+1} \leq d_m + p+1$.
\end{enumerate}
\end{pro}

\begin{proof} The first part is derived from Mori's Bend-and-Break and the properness of $\K_x$. Hence the second is trivial.
\end{proof}

\begin{defi}[{\cite[Subsection~2.2]{HK}}]\label{defil} \rm From Lemma~\ref{dim}~$\rm (ii)$, there exists an integer $m>0$ satisfying $d_m=n$ and $d_{m-1}<n$. We denote such $m$ by $l_{\K}$ and call {\it length} with respect to $\K$.
\end{defi}

\begin{rem}\label{remark} \rm 
\begin{enumerate}
\item From Lemma~\ref{dim}~$\rm (ii)$, we have $l_{\K} \leq n$.
\item When $p=0$, we know $l_{\K}=n$ from the above $\rm (i)$ and Proposition~\ref{in}. In this case, it is easy to see that this holds without the assumption of the irreducibility of $\K_x$.
\end{enumerate}
\end{rem}

\section{Main Theorem}

Continuously, we always work under the Assumptions~\ref{NA} and use notation as in the previous section. 

\begin{pro}\label{key} Let $X$ and $\K$ be as in the Notation-Assumptions~\ref{NA}.  If $d_{m+1}=d_m+1$, we have $d_{m+1}=n$.
\end{pro}

\begin{proof} We have $\widetilde{W^m_x} \cap \iota_0^{-1}(X^0) \subset W_x^{m+1} \subset \widetilde{W^{m+1}_x}$. Furthermore we know $\dim \widetilde{W^m_x}=d_m + p+1$ and $\dim W_x^{m+1}=d_{m+1}+p=d_m+p+1$ from Lemma~\ref{dim} $\rm (iii)$ and our assumption. For a subset $A \subset \U^0$, denote by $\overline{A}$ the closure of $A$ in $\U^0$. 
We see that $\overline{\widetilde{W^m_x} \cap \iota_0^{-1}(X^0)} \cap W_x^{m+1} =W_x^{m+1}$. Hence we have 
\begin{eqnarray*}
\widetilde{W^{m+1}_x} &=& \pi_0^{-1}(\pi_0(W_x^{m+1}))=\pi_0^{-1}(\pi_0((\overline{\widetilde{W^m_x} \cap \iota_0^{-1}(X^0)}) \cap W_x^{m+1})) \\ 
&\subset&  \pi_0^{-1}(\overline{\pi_0(\widetilde{W^m_x} \cap \iota_0^{-1}(X^0))} \cap \pi_0(W_x^{m+1})) \\ 
&=& \pi_0^{-1}(\overline{\pi_0(\widetilde{W^m_x} \cap \iota_0^{-1}(X^0))}) \cap \pi_0^{-1}(\pi_0(W_x^{m+1})) \\ 
&=& \overline{\pi_0^{-1}(\pi_0(\widetilde{W^m_x} \cap \iota_0^{-1}(X^0)))} \cap \widetilde{W^{m+1}_x}.
\end{eqnarray*}
Here the last equality holds because $\pi_0$ is an open morphism. Moreover we see that $\pi_0^{-1}(\pi_0(\widetilde{W^m_x} \cap \iota_0^{-1}(X^0))) \subset \widetilde{W^m_x}$. Therefore we have $\widetilde{W^{m+1}_x} \subset \overline{\widetilde{W^m_x}}$. This implies that $\overline{\widetilde{W^m_x}}=\overline{\widetilde{W^{m+1}_x}}$. Hence we obtain that $\dim \widetilde{W^m_x} = \dim \widetilde{W_x^{m+1}}$. Thus we see $d_{m+1}= d_{m+2}$. As a consequence, we have $d_{m+1}=n$ by Lemma~\ref{dim}~$\rm (ii)$.
\end{proof}

\begin{proof}[Proof of Theorem~\ref{MT}] Obviously, it follows from the definition that $d_{l_{\K}}=n$. Proposition~\ref{in} and \ref{key} imply that 
\begin{eqnarray} 
(p+1)+2(m-1) \leq d_m \leq m(p+1)~{\rm for}~m < l_{\K}. \nonumber
\end{eqnarray} 
When $d_m=(p+1)+2(m-1)$ for any $m<l_{\K}$, we have $n=d_{l_{\K}}=(p+1)+2(l_{\K}-1)$ or $(p+1)+2(l_{\K}-2)+1$. In this case, $l_{\K}=\lfloor \frac{n-p}{2} \rfloor+1$. On the other hand, when $d_m=m(p+1)$ for any $m<l_{\K}$, we have $n=d_{l_{\K}}=(l_{\K}-1)(p+1)+k$ for $1 \leq k \leq p+1$. In this case, $l_{\K}=\lfloor \frac{n-1}{p+1}\rfloor +1$. Hence our assertion holds. 
\end{proof}

\begin{cor}\label{sp} Let $X$ and $\K$ be as in the Notation-Assumptions~\ref{NA}. $l_{\K}$ is equal to $\lfloor \frac{n-1}{p+1} \rfloor+1= \lfloor \frac{n-p}{2}\rfloor +1$ if and only if one of the following holds:\\ 
${\rm (i)}~ n-3 \leq p \leq n-1$, ${\rm (ii)}~ p=1$, or ${\rm (iii)}~ (n,p)=(7,2)$.
\end{cor}

\begin{proof} The "if" part is derived from Theorem~\ref{MT}. The "only if" part follows from a direct computation. 
\end{proof}

\section{Varieties of minimal rational tangents and their secant varieties}

\subsection{Basic facts of varieties of minimal rational tangents}

Assume that $X$ is a Fano $n$-fold of $\rho=1$ (or more generally, a uniruled manifold) and $\K$ a minimal rational component of $X$ such that $p=\dim \K_x$ for general $x \in X$. It is {\it not} necessary to suppose the Assumption~\ref{NA}. Denote by $\widetilde{\K_x}$ the normalization of $\K_x$. Then it is known that $\widetilde{\K_x}$ is smooth for general $x \in X$ (see \cite[Theorem~1.3]{Hw2}). 
For a general point $x \in X$, we define the tangent map ${\tau}_x : \widetilde{{\K}_x} \rightarrow \PP(T_xX)$ by assigning the tangent vector at $x$ to each member of $\widetilde{\K_x}$ which is smooth at $x$. The regularity of $\tau_x$ follows from \cite[Theorem~3.4]{Ke2}. We denote by $\C_x \subset \PP(T_xX)$ the image of ${\tau}_x$, which is called the {\it variety of minimal rational tangents} at $x$. Let $\PP(W_x)$ be the linear span of $\C_x \subset \PP(T_xX)$ and $W$ the distribution defined by $W_x$ for general $x \in X$ (see \cite[Section~2]{Hw2}). Hwang's survey \cite{Hw2} is a standard reference on varieties of minimal rational tangents.

\begin{them} [{\cite[Theorem~1]{HM2}},{\cite[Theorem~3.4]{Ke2}}]\label{norm} Let $X$ be a Fano manifold (or more generally, a uniruled manifold). Then  the tangent map ${\tau}_x : \widetilde{{\K}_x} \rightarrow \C_x \subset \PP(T_xX)$ is the normalization.
\end{them}

\begin{them}[\cite{CMSB,Ke1}]\label{CMSB} Let $X$ be a Fano $n$-fold (or more generally, a uniruled $n$-fold). If $p=n-1$, namely $\C_x =\PP(T_xX)$, then $X$ is isomorphic to $\PP^n$.
\end{them}

\begin{them}[\cite{HH}]\label{HHM} Let $X$ be a Fano $n$-fold of $\rho=1$. Let $S=G/P$ be a rational homogeneous variety corresponding to a long simple root and $\C_o \subset \PP(T_oS)$ the variety of minimal rational tangents at a reference point $o \in S$. Assume $\C_o \subset \PP(T_oS)$ and $\C_x \subset \PP(T_xX)$ are isomorphic as projective subvarieties. Then $X$ is isomorphic to $S$.
\end{them}

\begin{cor}\label{LG} Let $X$ be a Fano $n$-fold of $\rho=1$. If $\C_x \subset \PP(T_xX)$ is projectively equivalent to the Veronese surface $v_2(\PP^2) \subset \PP^5$, $X$ is the $6$-dimensional Lagrangian Grassmann $LG(3,6)$ which parametrizes $3$-dimensional isotropic subspaces of a symplectic vector space $\CC^6$.
\end{cor}

\begin{proof} $LG(3,6)$ is the rational homogeneous variety corresponding to the unique long simple root of the Dynkin diagram $C_3$. Furthermore the variety of minimal rational tangents of $LG(3,6)$ at a general point is projectively equivalent to $v_2(\PP^2)$ (for example, see \cite[Proposition~1]{HM}). Hence Theorem~\ref{HHM} implies that $X$ is isomorphic to $LG(3,6)$.

\end{proof}

\begin{them}[\cite{Mi}]\label{Mi} Let $X$ be a Fano $n$-fold of $\rho=1$. If $n \geq 3$, the following are equivalent.
\begin{enumerate}
\item{$X$ is isomorphic to a smooth quadric hypersurface $Q^n$.}
\item{The minimal value of the anticanonical degree of rational curves passing through a general point $x_0 \in X$ is equal to $n$.}
\end{enumerate}
\end{them}

\begin{cor}\label{Mi2} Let $X$ be a Fano $n$-fold of $\rho=1$. If $p=n-2 \geq 1$, namely $\C_x \subset \PP(T_xX)$ is a hypersurface, $X$ is isomorphic to $Q^n$.
\end{cor}

\begin{proof} From our assumption $p=n-2$, $X$ is covered by rational curves of anticanonical degree $\leq n$. Since finitely many families of rational curves of anticanonical degree $<n$ cannot be dominating under the assumption $p=n-2$. Therefore $X$ is isomorphic to $Q^n$ by Theorem~\ref{Mi}.
\end{proof}

\begin{pro}[{\cite[Proposition~5]{Hw3}}]\label{lin} Let $X$ be a Fano $n$-fold of $\rho=1$. Then $\C_x \subset \PP(T_xX)$ cannot be a linear subspace except $\C_x = \PP(T_xX)$. 
\end{pro}

\begin{pro}[{\cite[Proposition~16]{A2}}]\label{ara} Let $X$ be a Fano manifold (or more generally, a uniruled manifold). If $\C_x \subset \PP(W_x)$ is an irreducible hypersurface for general $x \in X$, $W \subset T_X$ is integrable.
\end{pro}

\begin{pro}[{\cite[Proposition 2]{Hw1}}]\label{hwlem2} Let $X$ be a Fano $n$-fold of $\rho=1$. $W$ is integrable if and only if $W_x$ coincides with $T_xX$ for general $x \in X$.  
\end{pro}

\begin{pro}[cf. {\cite[Proposition~2.4 and 2.6]{Hw2}}]\label{mok} Let $X$ be a Fano $n$-fold of $\rho=1$. Assume that $\C_x$ is smooth and irreducible. If $2(p+1)> \dim W_x$ holds, $W \subset T_X$ is integrable.
\end{pro}

\subsection{Secant variety}

\begin{defi}{\rm For varieties $Z_1, Z_2 \subset \PP^N$, we define the {\it join} $S(Z_1,Z_2) \subset \PP^N$ by the closure of the union of lines connecting two distinct points $x_1 \in Z_1$ and $x_2 \in Z_2$. In the special case that $Z=Z_1=Z_2$, $SZ:=S(Z,Z)$ is called the {\it secant variety} of $Z$. }
\end{defi}

\begin{pro}[{\cite[Corollary 2.3.7]{Ru}}, {\cite{Se}}]\label{russo} Let $Z \subset \PP^N$ be an irreducible nondegenerate variety of dimension $n \geq 2$. Assume that $\dim SZ = n+2< N$. Then $Z$ is projectively equivalent to one of the following:
\begin{enumerate}
\item $Z \subset \PP^N$ is a cone over a curve, or
\item $Z \subset \PP^{n+3}$ is a cone over the Veronese surface $v_2(\PP^2) \subset \PP^5$ (When $n=2$, then $Z=v_2(\PP^2) \subset \PP^5$).
\end{enumerate}
\end{pro}

\begin{lem}[{\cite[Lemma~4.3]{A1}}]\label{arau} Let $Z \subset \PP^N$ be an irreducible cone whose normalization is smooth. Then $Z \subset \PP^N$ is a linear space.
\end{lem}

\subsection{Varieties of minimal rational tangents admitting low dimensional secant varieties}

\begin{pro}\label{p+1} Let $X$ and $\K$ be as in the Notation-Assumptions~\ref{NA}. Denote by $\C_x$ the variety of minimal rational tangents at a general point $x \in X$. 
\begin{enumerate}
\item If $\dim S\C_x=p$ for general $x \in X$, then $X=\PP^n$.  
\item If $\dim S\C_x=p+1$ for general $x \in X$, then $X=Q^n$.
\end{enumerate}
\end{pro}

\begin{proof} $\rm (i)$ Assume $\dim S\C_x=p$ for general $x \in X$. Then $\C_x=S\C_x \subset \PP(T_xX)$ is linear. Proposition~\ref{lin} implies that $\C_x=\PP(T_xX)$. Hence we have $p=n-1$. By Theorem~\ref{CMSB}, we see $X$ is isomorphic to $\PP^n$.\\ 
$\rm (ii)$ Assume $\dim S\C_x=p+1$ for general $x \in X$. Then, for $z \in S\C_x \setminus \C_x$, $S(z,\C_x)$ coincides with $S\C_x$. So we see $S(z,S\C_x)=S(z,S(z,\C_x))=S(z,\C_x)=S\C_x$. This implies that $S\C_x$ is a $(p+1)$-dimensional linear subspace. Thus we see $\C_x \subset \PP(W_x)=S\C_x$ is an irreducible hypersurface for general $x \in X$. From Proposition~\ref{ara}, $W \subset T_X$ is integrable. However Proposition~\ref{hwlem2} implies that $W_x$ coincides with $T_xX$ for general $x \in X$. Therefore we have $p=n-2$. It follows that $X$ is isomorphic to $Q^n$ from Corollary~\ref{Mi2}.
\end{proof}

\begin{pro}\label{p+2} Under the same assumption as in Proposition~\ref{p+1}, if $\dim S\C_x=p+2$ for general $x \in X$, then one of the following holds: 
\begin{enumerate}
\item $p=1$, 
\item $p=n-3$, 
\item $X$ is the Lagrangian Grassmann $LG(3,6)$, 
\item $\C_x$ is the Veronese surface in its linear span $\PP(W_x)=\PP^5$ and  $n>6$, or 
\item $\C_x$ is a degenerate singular variety satisfying $S\C_x=\PP^{p+2}$.
\end{enumerate}
\end{pro}

\begin{proof} Suppose that $\rm (i)$, $\rm (ii)$ and $\rm (iii)$ do not hold. Then it is enough to show that either $\rm (iv)$ or $\rm (v)$ hold. 

First assume that $S\C_x$ does not coincide with $\PP(W_x)$. From Proposition~\ref{russo}, Theorem~\ref{norm} and Lemma~\ref{arau}, $\C_x$ is projectively equivalent to the Veronese surface $v_2(\PP^2)$. If $\C_x \subset \PP(T_xX)$ is nondegenerate, $X$ is isomorphic to the Lagrangian Grassmann $LG(3,6)$ by Corollary~\ref{LG}. It contradicts our assumption. Hence $\C_x$ is degenerate. 

Second assume that $S\C_x=\PP(W_x)$. Note that $W_x$ does not coincide with $T_xX$ because $p$ is not $n-3$ by our assumption. Hence $\C_x \subset \PP(T_xX)$ is degenerate. Here we have $2(p+1)> \dim W_x$. If $\C_x$ is smooth, Proposition~\ref{mok} implies that the distribution $W \subset T_X$ is integrable. Furthermore we see $W_x=T_xX$ by Proposition~\ref{hwlem2}. It is a contradiction. Thus $\C_x$ is singular.

\end{proof}

\section{Low dimensional case}

Let $X$ be a Fano $n$-fold of $\rho=1$, $\K$ a minimal rational component of $X$ satisfying the Notation-Assumptions~\ref{NA}. We use the notation introduced in Section~\ref{Dl}.

\begin{them}[{\cite[Theorem~3.12]{HK}}]\label{HKS} $d_2 \geq \dim S\C_x+1$.
\end{them}

\begin{lem}\label{p+3} If $l_{\K}=\lfloor \frac{n-p}{2}\rfloor +1$, we have $\dim S\C_x \leq p+3$. Moreover if $\dim S\C_x=p+3$, then $n$ and $p+1$ are congruent modulo $2$.
\end{lem}

\begin{proof} Suppose that $l_{\K}=\lfloor \frac{n-p}{2}\rfloor +1$ holds. According to Proposition~\ref{key}, we know $d_2 +2(l_{\K}-3) +1 \leq d_{l_{\K}}=n$.  Hence we have $d_2 \leq n-2l_{\K}+5=n-2\lfloor \frac{n-p}{2}\rfloor +3$. The right hand side is equal to $p+3$ or $p+4$. Furthermore if it is $p+4$, then $n$ and $p+1$ are congruent modulo $2$. Consequently, our assertion follows from Theorem~\ref{HKS}.
\end{proof}

From Theorem~\ref{MT} (cf. Corollary~\ref{sp}) and Remark~\ref{remark}~{\rm (ii)}, we can compute the length $l_{\K}$ in the case $n \leq 7$. In fact, we obtain the following table:

\begin{center}
\begin{tabular}{|c|c|c||c|c|c||c|c|c||c|c|c||c|c|c|}
\hline
 $n$ & $p$ & $l_{\K}$ & $n$ & $p$ & $l_{\K}$ & $n$ & $p$ & $l_{\K}$ & $n$ & $p$ & $l_{\K}$ & $n$ & $p$ & $l_{\K}$ \\ \hline \hline
 $3$&$2$&$1$& $4$&$3$&$1$& $5$&$4$&$1$& $6$&$5$&$1$& $7$&$6$&$1$  \\
 $3$&$1$&$2$& $4$&$2$&$2$& $5$&$3$&$2$& $6$&$4$&$2$& $7$&$5$&$2$      \\
 $3$&$0$&$3$& $4$&$1$&$2$& $5$&$2$&$2$& $6$&$3$&$2$& $7$&$4$&$2$ \\
 $$ & $$&$$ & $4$&$0$&$4$& $5$&$1$&$3$& $6$&$2$&$2~{\rm or}~3$& $7$&$3$&$2~{\rm or}~3$  \\
 $$ & $$& $$& $$ &$$ & $$& $5$&$0$&$5$& $6$&$1$&$3$& $7$&$2$&$3$ \\
 $$ & $$& $$& $$ &$$ & $$& $$ &$$ &$$ & $6$&$0$&$6$& $7$&$1$&$4$ \\ 
 $$ & $$& $$& $$ &$$ & $$& $$ &$$ &$$ & $$& $$&$$&   $7$&$0$&$7$ \\ 

   \hline
\end{tabular}

\end{center}

Here we assume the irreducibility of $\K_x$ if $p \geq 1$. However $l_{\K}=n$ holds without the assumption of the irreducibility of $\K_x$ if $p=0$. 
From this table, we see the length $l_{\K}$ depends only on the pair $n$ and $p$ in the case $n \leq 5$. However it does not hold when $n \geq 6$. In fact, we have the following examples.

\begin{exa}
\begin{enumerate}
\item A $6$-dimensional smooth hypersurface of degree $4$ satisfies $(n,p,l_{\K})=(6,2,2)$. 
\item The Lagrangian Grassmann $LG(3,6)$ satisfies $(n,p,l_{\K})=(6,2,3)$. 
\end{enumerate}
\end{exa}

\begin{proof} See \cite[Proposition~6.2, Corollary~6.6]{HK}.
\end{proof}

 Here we study the structure of $X$ when $(n,p,l_{\K})=(6,2,3)$ and $(7,3,3)$.

\begin{pro} When $(n,p,l_{\K})=(6,2,3)$, one of the following holds:
\begin{enumerate}
\item $X=LG(3,6)$, or 
\item $\C_x \subset \PP(T_xX)$ is a degenerate singular surface satisfying $S\C_x=\PP^{4}$.
\end{enumerate}
\end{pro}

\begin{proof} From Lemma~\ref{p+3}, we have $\dim S\C_x \leq 4$. Hence our assertion is derived from Proposition~\ref{p+1} and \ref{p+2}.
\end{proof}

The same argument implies the following:

\begin{pro} When $(n,p,l_{\K})=(7,3,3)$, then $\C_x \subset \PP(T_xX)$ is a degenerate singular $3$-fold satisfying $S\C_x=\PP^{5}$.
\end{pro}

\begin{rem} \rm In general, it is believed that any variety of minimal rational tangents $\C_x$ at a general point is smooth.
\end{rem}

{\bf Acknowledgements}
The author would like to thank Professor Hajime Kaji for valuable seminars and encouragements. He is also grateful to referees, for their careful reading of the text and useful suggestions and comments. In particular, one of referees pointed out a gap of the proof of Lemma~\ref{dim}. 
The author is supported by Research Fellowships of the Japan Society for the Promotion of Science for Young Scientists.

\end{document}